\documentclass[a4paper,12pt]{article}
\usepackage{amssymb}
\usepackage{latexsym}
\usepackage{amsmath,amsthm,amssymb}
\usepackage{amsfonts}

\newtheorem{tw}{Theorem}[section]

\newtheorem{cor}[tw]{Corollary}

\theoremstyle{definition}

\begin{document}
\begin{center}

{\Large Fej\'er's approximation of continuous functions of unitary operators}\\
{\sc Krzysztof Zajkowski}\\
Institute of Mathematics, University of Bialystok \\ 
Akademicka 2, 15-267 Bialystok, Poland \\ 
E-mail:kryza@math.uwb.edu.pl 
\end{center}

\begin{abstract}
This paper is concerned with a certain aspect of the spectral theory of unitary operators in a Hilbert space and its aim is to give an explicit construction of continuous functions of unitary operators. Starting from a given unitary operator we give a family of sequences of trigonometric polynomials converging weakly to the complex measures which allow us to define functions of the operator.
\end{abstract}



\section{Introduction}

Spectral theorems belong to classical subjects of mathematics and they serve, among other purposes, as a tool to establish functional calculus, i.e. to define functions of operators. However, it seems that while different versions of spectral theorems define functions of operators, they rather seldom provide their explicit constructions. This work is intended as an attempt to find a more constructive method to define continuous functions of unitary operators. Using the Fej\'er theorem on approximation of continuous functions by Fourier series we can obtain the approximation of continuous functions of any unitary operator.

\section{Continuous functions of unitary operators}
One can approximate continuous functions in many way. The Fej\'er theorem says about one of them.
For every continuous $2\pi$-periodic function $f$ a sequence of trigonometric polynomials
$
\int_{[0,2\pi)}K_N(t-s)f(s)ds,
$
where $K_N(\tau)=\sum_{k=-N}^{k=N}(1-\frac{|k|}{N+1})e^{ik\tau}$ is the $N$-th Fej\'er kernel, uniformly converges
to the function $f$. A consequence of the Fej\'er theorem is a convergence of a sequence of trigonometric polynomials
$
\int_{[0,2\pi)}K_N(t-s)\mu(ds),
$
where $\mu$ denotes any finite (periodic) Borel measure $\mu$ on the interval $[0,2\pi)$, to the measure $\mu$
in the weak-$\ast$ topology. The above statement may serve to establish a base of the functional calculus for unitary operators.

Let $\mathcal{H}$ denote a complex Hilbert space with $\left\langle \cdot,\cdot\right\rangle$ being the inner product and $U$ any
unitary operator on $\mathcal{H}$. We prove that for every $x,y\in \mathcal{H}$ there exists a finite Borel measure $\mu_{x,y}$ on the interval $[0,2\pi)$ such that the expressions $\left\langle U^kx,y\right\rangle$, $k\in\mathbb{Z}$, will be equal to $k$-th trigonometric moments of this measure $\mu_{x,y}$, i.e.
\begin{equation*}
\int_{[0,2\pi)}e^{ikt}\mu_{x,y}(dt)=\left\langle U^kx,y\right\rangle.
\end{equation*}
Let $T_N(t)$ denote a one-parameter  family of operators
$$
T_N(t)=\sum_{n=0}^Ne^{-itn}U^n. 
$$
Notice now that 
\begin{equation*}
\frac{1}{N+1}\left\langle T_N(t)x,T_N(t)y\right\rangle=\sum_{k=-N}^{k=N}\Big(1-\frac{|k|}{N+1}\Big)\left\langle U^kx,y\right\rangle e^{-ikt}.
\end{equation*}
If there exists the above mentioned measure $\mu_{x,y}$ then this expression is equal to 
$
\int_{[0,2\pi)}K_N(t-s)\mu_{x,y}(ds).
$

Let $C(T)$ denote the Banach space of complex continuous functions on the unit circle $T$ endowed with
the supremum norm $\Vert\cdot\Vert_\infty$. If $f\in C(X)$ then $f(e^{it})$ is a complex continuous $2\pi$-periodic 
function on $\mathbb{R}$.
We define a functional $F^N_{x,y}$ on $C(T)$ as follows
\begin{equation*}
F^N_{x,y}(f)=\frac{1}{N+1}\frac{1}{2\pi}\int_{[0,2\pi)}\left\langle T_N(t)x,T_N(t)y\right\rangle f(e^{it})dt.
\end{equation*}
Now we prove the following theorem.
\begin{tw}
\label{tw1}
For any fixed $x,y\in\mathcal{H}$ the sequence of functionals $F^N_{x,y}$ weak-$\ast$ converges with $N\to\infty$
to a complex Borel measure. Denoting this limit measure by $\mu_{x,y}$ we have for every $f$ in $C(T)$
\begin{equation*}
\lim_{N\to\infty}F^N_{x,y}(f)=\int_{[0,2\pi)}f(e^{it})\mu_{x,y}(dt).
\end{equation*}
\end{tw}
\begin{proof}
Taking $f(z)=z^n$ for $z\in T$ ($n\in{\Bbb Z}$) we obtain 
\begin{eqnarray*}
F^N_{x,y}(z^n) & = & \frac{1}{N+1}\frac{1}{2\pi}\int_{[0,2\pi)}\left\langle T_N(t)x,T_N(t)y\right\rangle e^{int}dt\\
\; & = & \sum_{k=-N}^{k=N}\Big(1-\frac{|k|}{N+1}\Big)\left\langle U^kx,y\right\rangle \frac{1}{2\pi}\int_{[0,2\pi)}e^{i(n-k)t}dt\\
\; & = & \Big(1-\frac{|n|}{N+1}\Big)\left\langle U^n x,y\right\rangle,\;\;|n|\le N.
\end{eqnarray*}
For $n=0$ it gives
$$
F^N_{x,y}(1)=\frac{1}{N+1}\frac{1}{2\pi}\int_{[0,2\pi)}\left\langle T_N(t)x,T_N(t)y\right\rangle dt=
\left\langle x,y\right\rangle,
$$
in particular
\begin{equation}
\label{eq.1}
F^N_{x,x}(1)=\frac{1}{N+1}\frac{1}{2\pi}\int_{[0,2\pi)}\Vert T_N(t)x\Vert^2_\mathcal{H}dt=\Vert x\Vert^2_\mathcal{H}.
\end{equation}

Notice now that
\begin{equation}
\label{eq.2}
\lim_{N\to\infty}F^N_{x,y}(z^n)=\lim_{N\to\infty}\Big(1-\frac{|n|}{N+1}\Big)
\left\langle U^nx,y\right\rangle=\left\langle U^nx,y\right\rangle
\end{equation}
for each $n\in{\Bbb Z}$.

The norms of the functionals $F^N_{x,y}$
$$
\Vert F^N_{x,y} \Vert \le \frac{1}{N+1}\frac{1}{2\pi}\int_{[0,2\pi)}|\left\langle T_N(t)x,T_N(t)y\right\rangle|dt.
$$
Applying twice the Schwarz inequality and (\ref{eq.1}) we obtain
\begin{eqnarray*}
\Vert F^N_{x,y} \Vert & \le &  \frac{1}{N+1}\frac{1}{2\pi}\int_{[0,2\pi)}\Vert T_N(t)x\Vert_\mathcal{H}\Vert T_N(t)y\Vert_\mathcal{H}dt\\
\; & \le & \Big(\frac{1}{N+1}\frac{1}{2\pi}\int_{[0,2\pi)}\Vert T_N(t)x\Vert^2_\mathcal{H}dt\Big)^\frac{1}{2}
\Big(\frac{1}{N+1}\frac{1}{2\pi}\int_{[0,2\pi)}\Vert T_N(t)y\Vert^2_\mathcal{H}dt\Big)^\frac{1}{2}\\
\; & = & \Vert x \Vert_\mathcal{H}\Vert y \Vert_\mathcal{H}.
\end{eqnarray*}
The sequence ($z^n$) is linearly dense in $C(T)$ and norms of $F^N_{x,y}$ are bounded by 
$\Vert x \Vert_\mathcal{H}\Vert y \Vert_\mathcal{H}$. It follows that 
$\lim_{N\to\infty}F^N_{x,y}(f)$ is linear bounded functional on $C(X)$. By the Riesz representation
theorem we obtain that
$$
\lim_{N\to\infty}F^N_{x,y}(f)=\int_{[0,2\pi)}f(e^{it})\mu_{x,y}(dt),
$$
where $\mu_{x,y}$ denotes a complex Borel measure on $[0,2\pi)$. 
\end{proof}
Notice now that the relation (\ref{eq.2}) extends on any trigonometric polynomial $p(e^{it})=\sum_{k=-N}^Nc_ke^{ikt}$
\begin{equation}
\label{eq.3}
\lim_{N\to\infty}F^N_{x,y}(p)=\left\langle p(U)x,y\right\rangle,
\end{equation}
where $p(U)=\sum_{k=-N}^Nc_k U^k$. Now we prove that this limiting relation extends on all continuous $2\pi$-periodic function.
\begin{tw}
\label{tw2}
For every continuous function $f\in C(T)$ there exists a bounded operator, denoted by $f(U)$, such that
$$
\lim_{N\to\infty}F^N_{x,y}(f)=\left\langle f(U)x,y\right\rangle.
$$ 
\end{tw}
\begin{proof}
By (\ref{eq.3})  and Theorem \ref{tw1} we have 
\begin{equation}
\label{eq.4}
\int_{[0,2\pi)}p(e^{it})\mu_{x,y}(dt)=\left\langle p(U)x,y\right\rangle
\end{equation}
and the following estimation $\Vert p(U)\Vert\le \Vert p \Vert_\infty$. Let now a sequence of trigonometric polynomials $p_n$ converge uniformly to a continuous function $f$ on $T$. Then the sequence $p_n(U)$ also converges to a limit
which we will denote by $f(U)$. By continuity of the inner product we get
$$
\lim_{n\to\infty}\left\langle p_n(U)x,y\right\rangle=\left\langle f(U)x,y\right\rangle.
$$
But on the other hand by continuity of the Borel measure and the equation (\ref{eq.4})  we obtain
\begin{eqnarray*}
\lim_{n\to\infty}\left\langle p_n(U)x,y\right\rangle & = & \lim_{n\to\infty}\int_{[0,2\pi)}p_n(e^{it})\mu_{x,y}(dt)\\
\; & = & \int_{[0,2\pi)}f(e^{it})\mu_{x,y}(dt).
\end{eqnarray*}
In this way by Theorem \ref{tw1} and the above two equations we obtained
$$
\lim_{N\to\infty}F^N_{x,y}(f)=\int_{[0,2\pi)}f(e^{it})\mu_{x,y}(dt)=\left\langle f(U)x,y\right\rangle.
$$
\end{proof}
\begin{cor}
For any unitary operator $U$ on the Hilbert space $\mathcal{H}$ the limiting relation of Theorem \ref{tw2}
defines $\ast$-homomorphism of the $C^\ast$-algebra of continuous $2\pi$-periodic functions into $C^\ast$-algebra 
of bounded operators on $\mathcal{H}$.
\end{cor}
\begin{proof}
For two trigonometric polynomials $p_1$ and $p_2$ we have
\begin{eqnarray*}
\int_{[0,2\pi)}(p_1\cdot p_2)(e^{it})\mu_{x,y}(dt) & = & \left\langle (p_1\cdot p_2)(U)x,y\right\rangle\\
\; & = & \left\langle p_1(U)p_2(U)x,y\right\rangle.
\end{eqnarray*}
Similarly, as in the proof of Theorem \ref{tw2}, we can extend this equation to any continuous functions $f_1,f_2\in C(T)$
$$
\int_{[0,2\pi)}(f_1\cdot f_2)(e^{it})\mu_{x,y}(dt)=\left\langle f_1(U)f_2(U)x,y\right\rangle. 
$$
Hence we get
$$
\lim_{N\to\infty}F^N_{x,y}(f_1f_2)=\left\langle f_1(U)f_2(U)x,y\right\rangle.
$$

Moreover,
\begin{eqnarray*}
\lim_{N\to\infty}F^N_{x,y}(\overline{f}) & = & \lim_{N\to\infty}
\frac{1}{N+1}\frac{1}{2\pi}\int_{[0,2\pi)}\overline{\left\langle T_N(t)y,T_N(t)x\right\rangle f(e^{it})}dt\\
\; & = & \overline{\left\langle f(U)y,x\right\rangle}=\left\langle f(U)^\ast x,y\right\rangle.
\end{eqnarray*}
This means, that the limiting relation of  Theorem \ref{tw2} define $\ast$-homomorphism of continuous $2\pi$-periodic functions into the $C^\ast$-algebra of bounded operators on $\mathcal{H}$.
\end{proof}

In this way we get the desired approximation of continuous functions of unitary operators.

{\bf Acknowledgment}. The author thanks Prof. A. Strasburger for his interest of this work and many interesting remarks.


\end{document}